\documentclass{amsart}
\usepackage{mathrsfs,amssymb,multirow,setspace,color}
\usepackage{hyperref}
\usepackage[a4paper, total={7in, 9in}]{geometry}
\theoremstyle{plain}
\usepackage{graphicx}
\newtheorem{theorem}{Theorem}[section]
\newtheorem{lemma}[theorem]{Lemma}

\newtheorem{corollary}[theorem]{Corollary}

\theoremstyle{definition}

\theoremstyle{remark}
\newtheorem{remark}[theorem]{Remark}

\input xy
\xyoption{all}
\def\a{\mathcal{P}_E(G)}
\def\c{\overline{\mathcal{P}_E(G)}}

\def\m{\mathcal{M}(G)}
\def\g{\mathcal{G}_{\m}}
\begin{document}
\title[Lambda  Number of the enhanced power graph of a finite group]{Lambda  Number of the enhanced power graph of a finite group}


\author[Parveen, Sandeep Dalal, Jitender Kumar]{Parveen, Sandeep Dalal, $\text{Jitender Kumar}^{^*}$ }
\address{$\text{}^1$Department of Mathematics, Birla Institute of Technology and Science Pilani, Pilani, India}
\address{$\text{}^2$School of Mathematical Sciences, National Institute of Science Education and Research, Bhubaneswar, Odisha, India}
\email{p.parveenkumar144@gmail.com,deepdalal10@gmail.com,jitenderarora09@gmail.com}

\begin{abstract}
 The enhanced power graph of a finite group $G$ is the simple undirected graph whose vertex set is $G$ and two distinct vertices $x, y$ are adjacent if $x, y \in \langle z \rangle$ for some $z \in G$. An $L( 2,1)$-labeling of graph $\Gamma$ is an integer labeling of $V(\Gamma)$ such that adjacent vertices have labels that differ by at least $2$ and vertices distance $2$ apart have labels that differ by at least $1$. The $\lambda$-number of $\Gamma$, denoted by $\lambda(\Gamma)$, is the minimum range over all $L( 2,1)$-labelings. In this article, we study the lambda number of the enhanced power graph $\a$ of the group $G$. This paper extends the corresponding results, obtained in [X. Ma, M. Feng, and K. Wang. Lambda number of the power graph of a finite group. \emph{J. Algebraic Combin.}, 53(3):743–754, 2021], of the lambda number of power graphs to enhanced power graphs. Moreover, for a non-trivial simple group $G$ of order $n$, we prove that $\lambda(\a) = n$ if and only if $G$ is not a cyclic group of order $n\geq 3$. Finally, we compute the exact value of $\lambda(\a)$ if $G$ is a finite nilpotent group.
 \end{abstract}


\keywords{Enhanced power graph, lambda number, maximal cyclic subgroup, nilpotent groups \\ *  Corresponding author}

\maketitle
\section{Introduction}
For non-negative integers $j$ and $k$, an $L(j, k)$-labeling for the graph $\Gamma$ is an integer valued function $f$ on the vertex set $V(\Gamma)$ such that $|f(u)-f(v)| \geq k$ whenever $u$ and $v$ are vertices of distance two and $|f(u)-f(v)| \geq j$ whenever $u$ and $v$ are adjacent. The span of $f$ is the difference between the maximum and minimum of $f$. It is convenient to assume that the minimum of $f$ is $0$ , we regard the span of $f$ as the maximum of $f$. The $L(j, k)$\emph{-labeling number} $\lambda_{j,k}(\Gamma)$ of the graph $\Gamma$ is the minimum span over all $L(j, k)$-labelings for $\Gamma$.  The classical work of the $L(j, k)$ labeling problem is when $j=2$ and $k=1$. The $L(2,1)$-labeling number of a graph $\Gamma$ is also called the $\lambda$\emph{-number} of $\Gamma$.

The radio channel assignment problem \cite{a.Hale1980} and the study of the scalability of optical networks \cite{a.Roberts1991}, motivated the researchers to investigate the problem related to $L(j, k)$-labelings of a graph. The concept of $L(2, 1)$-labeling of a graph was introduced by Griggs and Yeh \cite{a.Griggs1992}, and they showed that for a general graph, $L(2, 1)$-labeling is NP-complete. Georges and Mauro \cite{a.Georges1995} later presented a generalization of the concept. The $L(j, k)$-labeling in particular $L(2,1)$-labeling, has been studied extensively by various authors (see \cite{a.Georger2003,a.Georges1994,a.Kim2017,a.Whittlesey1995}). A survey of results and open problems related to the $L(j,k)$-labeling of a graph can be found in  \cite{a.Yeh2006}. 

Graphs associated with groups and other algebraic structures have been studied by various researchers as they have valuable applications and are related to the automata theory (cf. \cite{kelarev2003graph,a.kelarev2009cayley,kelarev2002ring,kelarev2004labelled}).  Zhou \cite{a.Zhou2005} investigated $L(j,k)$-labeling of Cayley graphs of abelian groups. Kelarev \emph{et al.} \cite{a.kelarev2015} showed connections between the structure of a semigroup and the minimum spans of distance labelings of its Cayley graphs. In 2021, Ma \emph{et al.} \cite{a.Malambda} studied the $L(2,1)$-labeling of the power graph of a finite group. Recently, Sarkar \cite{a.sarkar2022lambda} investigated the lambda number of power graph of finite simple group. Mishra \cite{a.mishra2021lambda}, studied the lambda number of power graphs of finite $p$-groups. 

In order to measure how much the power graph is close to the commuting graph of a group $G$, Aalipour \emph{et al.} \cite{a.Cameron2016} introduced a new graph called \emph{enhanced power graph} of the group $G$. The enhanced power graph of a group $G$ is the simple undirected graph whose vertex set is $G$ and two distinct vertices $x, y$ are adjacent if $x, y \in \langle z \rangle$ for some $z \in G$. Indeed, the enhanced power graph contains the power graph and is a spanning subgraph of the commuting graph. The study of enhanced power graphs have received the considerable attention by various researchers. Aalipour \emph{et al.} \cite{a.Cameron2016} characterized the finite group $G$, for which equality holds for either two of the three graphs viz. power graph, enhanced power graph and commuting graph of $G$.  Bera \emph{et al.} \cite{a.Bera2017} characterized the abelian groups and the non abelian $p$-groups having dominatable enhanced power graphs. A complete description of finite groups with enhanced power graphs admitting a perfect code have been studied in \cite{a.ma2017perfect}. Ma \emph{et al.} \cite{a.ma2020metric} investigated the metric dimension of an enhanced power graph of finite groups. Hamzeh \emph{et al.} \cite{a.hamzeh2017automorphism} derived the automorphism groups of enhanced power graphs of finite groups.  Zahirovi$\acute{c}$ \emph{et al.} \cite{a.zahirovic2020study} proved that two finite abelian groups are isomorphic if their enhanced power graphs are isomorphic. Also, they supplied a characterization of finite nilpotent groups whose enhanced power graphs are perfect.  Recently, Panda \emph{et al.} \cite{a.panda2021enhanced} studied the graph-theoretic properties, viz. minimum degree, independence number, matching number, strong metric dimension and perfectness, of enhanced power graphs over finite abelian groups. Moreover, the enhanced power graphs associated to non-abelian groups such as semidihedral, dihedral, dicyclic, $U_{6n}$, $V_{8n}$ etc., have been studied in \cite{a.dalal2021enhanced, a.panda2021enhanced}. Bera \emph{et al.} \cite{a.bera2021connectivity} gave an upper bound for the vertex connectivity of enhanced power graph of any finite abelian group. Moreover, they classified the finite abelian groups whose proper enhanced power graphs are connected. Results related to the connectivity, dominating vertices and the spectral radius of proper enhanced power graph have been investigated by Bera \emph{et al.} in \cite{a.bera2021connectivitydomin}.  Recently, other graph theoretic properties, namely: regularity, vertex connectivity and the Wiener index, of the enhanced power graphs of finite groups have been studied in \cite{a.parveen2022}. For a comprehensive list of results and open questions on enhanced power graphs of groups, we refer the reader to  \cite{a.ma2022survey}.

The lambda number of the power graphs of finite groups has been studied in \cite{a.Malambda,a.mishra2021lambda,a.sarkar2022lambda}. In this paper, we study the lambda number of the enhanced power graph of a finite group $G$. In Section $2$ we recall the necessary definitions, results and fixed our notations which we used throughout the paper. Section $3$ comprises the main results of the present paper.
\section{Preliminaries}
In this section, first we recall the graph theoretic notions from  \cite{b.westgraph}. A \emph{graph} $\Gamma$ is a pair  $\Gamma = (V, E)$, where $V(\Gamma)$ and $E(\Gamma)$ are the set of vertices and edges of $\Gamma$, respectively. Two distinct vertices $u_1$ and $ u_2$ are $\mathit{adjacent}$, denoted by $u_1 \sim u_2$, if there is an edge between $u_1$ and $u_2$. Otherwise, we write it as $u_1 \nsim u_2$. Let $\Gamma$ be a graph. A \emph{subgraph}  $\Gamma'$ of $\Gamma$ is the graph such that $V(\Gamma') \subseteq V(\Gamma)$ and $E(\Gamma') \subseteq E(\Gamma)$. A subgraph $\Gamma '$ of graph $\Gamma$ is said to be a \emph{spanning subgraph} of $\Gamma$ if $V(\Gamma) = V(\Gamma ')$. For $X \subseteq V(\Gamma)$, the subgraph of $\Gamma$ induced by the set $X$ is the graph with vertex set $X$ and its two distinct vertices are adjacent if and only if they are adjacent in $\Gamma$. The \emph{complement} $\overline{\Gamma}$ of $\Gamma$ is a graph with same vertex set as $\Gamma$ and distinct vertices $u, v$ are adjacent in $\overline{\Gamma}$ if they are not adjacent in $\Gamma$. A graph $\Gamma$ is said to be \emph{complete} if any two distinct vertices are adjacent. We denote $K_n$ by the complete graph of $n$ vertices. A graph $\Gamma$ is said to be $k$-partite if the vertex set of $\Gamma$ can be partitioned into $k$ subsets, such that no two vertices in the same subset of the partition are adjacent. A \emph{complete k-partite} graph, denoted by $K_{n_1,n_2,\ldots ,n_k}$, is a $k$-partite graph having its parts sizes $n_1,n_2,\ldots ,n_k$ such that every vertex in each part is adjacent to all the vertices of all other parts of $K_{n_1,n_2,\ldots ,n_k}$. A vertex $v$ of $\Gamma$ is said to be a \emph{dominating vertex} if $v$ is adjacent to all the other vertices of $\Gamma$. We denote $\mathrm{Dom}(\Gamma)$ by the set of all dominating vertices of the graph $\Gamma$.  A \emph{walk} $\lambda$ in $\Gamma$ from the vertex $u$ to the vertex $w$ is a sequence of vertices $u = v_1, v_2, \ldots , v_m = w (m > 1)$ such that $v_i \sim v_{i+1}$ for every $i \in \{1, 2, \ldots , m-1\}$.  A walk is said to be a \emph{path} if no vertex is repeated. A graph $\Gamma$ is \emph{connected}  if each pair of vertices has a path in $\Gamma$. Otherwise, $\Gamma$ is \emph{disconnected}.  The \emph{distance} between $u, v \in V(\Gamma)$, denoted by $d(u, v)$,  is the number of edges in a shortest path connecting them. A \emph{path covering} $\mathrm{C}(\Gamma)$ of a graph $\Gamma$ is a collection of vertex-disjoint paths in $\Gamma$ such that each vertex in $V(\Gamma)$ is contained in a path of $\mathrm{C}(\Gamma)$. The \emph{path covering number} $c(\Gamma)$ of $\Gamma$ is the minimum cardinality of a path covering of $\Gamma$.

Let $G$ be a group. The \emph{order of an element} $x$ in $G$ is the cardinality of the subgroup generated by $x$ and it is denoted by $o(x)$. For a positive integer $n$, $\phi(n)$ denotes the Euler's totient function of $n$. Consider the set $\pi _G=\{o(g): g\neq e \in G\}$. The \emph{exponent} of a group is defined as the least common multiple of the orders of all elements of the group. For any $x,y \in G$, define a relation $\rho$ such that $x\rho y$ if and only if $\langle x\rangle = \langle y \rangle$. Note that $\rho$ is an equivalence relation and the equivalence class of $x$ is denoted by $\rho _x$. For $d\in \pi _G$, $C_d$ denotes the number of equivalence classes that consists of elements of order $d$. Moreover, we denote $\tau _d =\{g\in G : o(g)=d\}$.  For $n \geq 3$, the \emph{dihedral group} $D_{2n}$ is a group of order $2n$ is defined in terms of generators and
relations as $D_{2n} = \langle x, y  :  x^{n} = y^2 = e,  xy = yx^{-1} \rangle$. For $n \geq 2$, the \emph{dicyclic group} $Q_{4n}$ is a group of order $4n$ is defined in terms of generators and
relations as $Q_{4n} = \langle a, b  :  a^{2n}  = e, a^n= b^2, ab = ba^{-1} \rangle.$ A cyclic subgroup of a group $G$ is called a \emph{maximal cyclic subgroup} if it is not properly contained in any cyclic subgroup of $G$ other than itself. If $G$ is a cyclic group, then $G$ is  the only maximal cyclic subgroup of $G$. We denote $\m$ by the set of all maximal cyclic subgroups of $G$. Also, $M\in \m$, we write $\mathcal{G} _M=\{x\in G : \langle x \rangle =M \}$ and $\g=\{x\in G : \langle x \rangle \in \m \}$.
Let $G$ be a group and $H, K $ be subgroups of $G$. The subgroup $[H, K] $ of $ G$ is defined as the subgroup generated by all elements of the form $[h, k]:=h^{-1} k^{-1} h k$, where $h \in H, k \in K$. The lower central series of subgroups of $G$ is the descending sequence
$$
G \geq G^{(2)} \geq G^{(3)} \geq \cdots \geq G^{(i)} \geq G^{(i+1)} \geq \cdots
$$
of normal subgroups of $G$ given by $G^{(2)}:=[G, G]$ and $G^{(i+1)}:=\left[G^{(i)}, G\right]$ for every $i \geq 2$. If for a group $G$, this descending sequence contains only finitely many non-trivial terms then $G$ is said to be a \emph{nilpotent} group. Every finite $p$-group is a nilpotent group.

A finite $p$-group of order $p^{n}$ is said to be of maximal class if $ G^{(n-1)} \neq\left\{e\right\}$ and $ G^{(n)}=\left\{e\right\}$. In this case, $G /  G^{(2)} \cong \mathbb{Z}_p \times \mathbb{Z}_p$ and $ G^{(i)} /  G^{(i+1)} \cong \mathbb{Z}_{p}$ for all $2 \leq i \leq n-1$. The following result says a lot more about the class numbers of a finite $p$-group $G$. 
\begin{theorem}{\rm \cite{a.pgroupberkovi,b.pgroupisaac2006,a.pgroupkulakoff,a.pgroupmiller}}{\label{pgroupclass}}
Let $G$ be a finite $p$-group of exponent $p^{k}$. Assume that $G$ is not cyclic for an odd prime $p$, and for $p=2$, it is neither cyclic nor of maximal class. Then
\begin{enumerate}
    \item[(i)] $C_p \equiv 1+p\left(\bmod p^{2}\right)$.
    \item[(ii)] $p \mid C_{p^i}$ for every $2 \leq i \leq k$.
\end{enumerate}
\end{theorem}
\begin{corollary}{\rm \cite{a.mishra2021lambda}}{\label{pgroupclasscorollary}}
Let $G$ be a finite $p$-group of exponent $p^{k}$. Then $C_{p^{i}}=1$ for some $1 \leq i \leq k$ if and only if one of the following occurs:
\begin{enumerate}
    \item $G \cong \mathbb{Z}_{p^{k}}$ and $C_{p^{j}}=1$ for all $1 \leq j \leq k$, or
    \item $p=2$ and $G$ is isomorphic to one of the following $2$-groups:
\end{enumerate}
\begin{enumerate}
    \item[(i)] dihedral $2$-group
$$
\mathbb{D}_{2^{k+1}}=\left\langle x, y: x^{2^{k}}=1, y^{2}=1, y^{-1} x y=x^{-1}\right\rangle, \quad(k \geq 1)
$$
where $C_2=1+2^{k}$ and $C_{2^{j}}=1 \text{ for all } (2 \leq j \leq k)$.
\item[(ii)] generalized quaternion $2$-group
$$
\mathbb{Q}_{2^{k+1}}=\left\langle x, y: x^{2^{k}}=1, x^{2^{k-1}}=y^{2}, y^{-1} x y=x^{-1}\right\rangle, \quad(k \geq 2)
$$
where $C_4=1+2^{k-1}$ and $C_{2^{j}}=1$ for all $1 \leq j \leq k$ and $j \neq 2$.
\item[(iii)] semi-dihedral $2$-group
$$
\mathbb{SD}_{2^{k+1}}=\left\langle x, y: x^{2^{k}}=1, y^{2}=1, y^{-1} x y=x^{-1+2^{k-1}}\right\rangle, \quad(k \geq 3)
$$
where $C_2=1+2^{k-1}, C_4=1+2^{k-2}$ and $C_{2^{j}}=1$ for all $3 \leq j \leq k$.
\end{enumerate}
\end{corollary}

\begin{theorem}{\rm \cite{b.dummit1991abstract}}{\label{nilpotent}}
 Let $G$ be a finite group. Then the following statements are equivalent:
 \begin{enumerate}
     \item[(i)] $G$ is a nilpotent group.
     \item[(ii)] Every Sylow subgroup of $G$ is normal.
    \item[(iii)] $G$ is the direct product of its Sylow subgroups.
    \item[(iv)] For $x,y\in G, \  x$ and $y$ commute whenever $o(x)$ and $o(y)$ are relatively primes.
 \end{enumerate}
 \end{theorem}
Any non-cyclic nilpotent group $G$ is of one of the following forms
\begin{enumerate}
    \item $G\cong G'\times \mathbb{Z}_n$, where $G'$ is a non-trivial nilpotent group of odd order having no cyclic Sylow subgroup and $\mathrm{gcd}(n,|G'|)=1$.
    \item $G\cong G'\times P\times \mathbb{Z}_n$, where $G'$ is a nilpotent group of odd order having no cyclic Sylow subgroup, $P$ is a $2$-group which is neither cyclic nor of maximal class and $\mathrm{gcd}(n,|G'|)= \mathrm{gcd}(2,n)=1$.
    \item $G\cong G'\times \mathbb{Q}_{2^{k+1}}\times \mathbb{Z}_n$, where $G'$ is described as in (2),  $\mathbb{Q}_{2^{k+1}}$ is a generalized quaternion group of order $2^{k+1}$ and $\mathrm{gcd}(n,|G'|)= \mathrm{gcd}(2,n)=1$.
    \item $G\cong G'\times \mathbb{D}_{2^{k+1}}\times \mathbb{Z}_n$, where $G'$ is described as in (2),  $\mathbb{D}_{2^{k+1}}$ is a dihedral group of order $2^{k+1}$ and $\mathrm{gcd}(n,|G'|)= \mathrm{gcd}(2,n)=1$.
     \item $G\cong G'\times \mathbb{SD}_{2^{k+1}}\times \mathbb{Z}_n$, where $G'$ is described as in (2),  $\mathbb{SD}_{2^{k+1}}$ is a dihedral group of order $2^{k+1}$ and $\mathrm{gcd}(n,|G'|)= \mathrm{gcd}(2,n)=1$.
\end{enumerate}
The following result characterizes the dominating vertices of enhanced power graph of a finite nilpotent group and we use this result explicitly in this paper without referring to it. 
\begin{theorem}{\rm \cite[Theorem 4.1]{a.bera2021connectivitydomin}}{\label{Dominating nilpotent}}
Let $G$ be a finite non-cyclic nilpotent group and let $D_{1}=\left\{\left(e', e_2,x\right): x \in \mathbb{Z}_{n}\right\}$, $D_{2}=\left\{(e', y, x):  y \in \mathbb{Q}_{2^{k+1}}, x \in \mathbb{Z}_{n} \right.$ and $\left.\mathrm{o}(y)=2\right\}$. Then
$$
\mathrm{Dom}(\a )= \begin{cases}\left\{(e', x): x \in \mathbb{Z}_{n}\right\}, & \text { if } G=G'\times \mathbb{Z}_{n}  \text { and } \operatorname{gcd}\left(\left|G'\right|, n\right)=1 \\ \left\{(e', e_1, x): x \in \mathbb{Z}_{n}\right\}, & \text { if } G=G'\times P\times \mathbb{Z}_n \text { and } \operatorname{gcd}\left(\left|G'\right|, n\right)  =\operatorname{gcd}(n, 2)=1 \\  D_{1} \cup D_{2}, & \text { if } G= G'\times \mathbb{Q}_{2^{k+1}}\times \mathbb{Z}_n \text { and } \operatorname{gcd}\left(\left|G'\right|, n\right)  =\operatorname{gcd}(n, 2)=1 \\ \left\{\left(e',  e_3, x\right): x \in \mathbb{Z}_{n}\right\} , & \text { if } G=G'\times \mathbb{D}_{2^{k+1}}\times \mathbb{Z}_n \text { and } \operatorname{gcd}\left(\left|G'\right|, n\right)  =\operatorname{gcd}(n, 2)=1 \\ \left\{\left(e', e_4,x\right): x \in \mathbb{Z}_{n}\right\} , & \text { if } G=G'\times \mathbb{SD}_{2^{k+1}}\times \mathbb{Z}_n \text { and } \operatorname{gcd}\left(\left|G'\right|, n\right)  =\operatorname{gcd}(n, 2)=1 ,\end{cases}
$$
where $e', e_i$'s, $1\leq i \leq 4$, are the identity elements of the respective groups in $G$.
\end{theorem}
\begin{lemma}{\label{lcm}}
Let $G=P_1\times P_2\times \cdots \times P_r$ be a finite nilpotent group of order $n=p_1^{\alpha _1}p_2^{\alpha _2}\cdots p_r^{\alpha _r}$. Suppose $x,y \in G$ such that $o(x)=s$ and  $o(y)=t$. Then there exists an element $z\in G$ such that $o(z)=\mathrm{lcm}(s,t)$.
\end{lemma}
\begin{proof}
Let $x=(x_1,x_2,\ldots ,x_r)$, $y=(y_1,y_2 ,\ldots ,y_r)\in G$. It follows that $s=p_1^{\beta _1}p_2^{\beta _2}\cdots p_r^{\beta _r}$ and $t=p_1^{\gamma _1}p_2^{\gamma _2}\cdots p_r^{\gamma _r}$, where $p_i^{\beta _i}=o(x_i)$ , $p_i^{\gamma _i}=o(y_i)$ and $0\leq \beta_i, \gamma_i \leq \alpha_i$. Consequently, $\mathrm{lcm}(s,t)=p_1^{\delta _1}p_2^{\delta _2}\cdots p_r^{\delta _r}$, where $\delta_i= \text{max}\{\beta_i, \gamma _i \}$. Consider $z=(z_1,z_2,\ldots ,z_r)$ such that $$z_i =  \left\{ \begin{array}{ll}
x_i & \mbox{if $\beta_i\geq \gamma_i$,}\\
y_i& \mbox{ if $\beta_i < \gamma_i$}. \end{array} \right. $$
Clearly, $z\in G$ and $o(z)=\prod \limits_{i=1} ^r o(z_i) =\prod \limits_{i=1} ^r p_i^{\delta_i}$. Thus, the result holds.
\end{proof}

\begin{remark}{\label{remark maximal}}
Let $G$ be a finite group. Then $G= \bigcup\limits_{M\in \m} M$ and the generators of a maximal cyclic subgroup does not belong to any other maximal cyclic subgroup of $G$. Consequently, if $|M_i|$ is a prime number then for distinct $M_i,  M_j \in \m$, we have $M_i \cap M_j =\{e\}$.
\end{remark}

The following results will be useful for further study.
\begin{lemma}{\label{2maximal}}
If $G$ is a finite group, then $| \m |\neq 2$.
\end{lemma}
\begin{proof}
On contrary, assume that the group $G$ has two maximal cyclic subgroups $M_1$ and $M_2$. Then every element of $G$ belongs to at least one of the maximal cyclic subgroup of $G$ and $e\in M_1\cap M_2$. It follows that 
$|M_1|+|M_2|\geq o(G)+1.$
Since $M_1$ and $M_2$ are proper subgroups of a finite group $G$, by Lagrange's theorem, we have $$|M_1|\leq \frac{o(G)}{2} \ \text{and} \  |M_2|\leq \frac{o(G)}{2}.$$
Consequently, we get $o(G)+1 \leq |M_1|+|M_2| \leq o(G)$, which is not possible. Hence, $| \m |\neq 2$.
\end{proof}
\begin{lemma}{\rm \cite[Lemma 2.1]{a.sarkar2022lambda}}{\label{cd2}}
Let $G$ be a finite non-cyclic simple group. Then for any $d\in \pi _G$, we have $C_d\geq 2$.
\end{lemma}
\begin{theorem}{\rm \cite[Theorem 2.4]{a.Bera2017}}{\label{complete}}
The enhanced power graph $\a$   of the group $G$  is complete if and only if $G$ is cyclic.
\end{theorem}
\begin{theorem}{\rm \cite[Theorem 14]{a.Georges1994}}{\label{Lambdain path cover}}
Let $\Gamma$ be a graph of order $n$.
\begin{itemize}
   \item[(i)] Then $\lambda(\Gamma)=n-1$ if and only if $c(\overline{\Gamma})=1$.
    \item[(ii)] Let $r$ be an integer at least $2$. Then $\lambda(\Gamma)=n+r-2$ if and only if $c(\overline{\Gamma})=r$.
\end{itemize}
\end{theorem}

\section{Main Results} 
In this section, we present the main results of the paper. Recall that if $G$ is a finite cyclic group then the enhanced power graph $\a$ is complete. First we obtain the bounds for $\lambda(\a)$, where $G$ is a finite group. Then we classify finite simple groups $G$ such that $\lambda(\a) = |G|$. Since the set of dominating vertices in $\a$ is known (see Theorem \ref{Dominating nilpotent}), we obtain the lambda number of the enhanced power graphs of nilpotent groups (see Theorem \ref{lambda number theorem odd} and Theorem \ref{lambda number main}).

\begin{theorem}{\label{lowerbound}}
Let G be a finite group of order $n$. Then $\lambda(\a)\geq n$ with equality holds if and only if $\overline{\a \setminus \{e\}}$ contains a Hamiltonian path.
\end{theorem}
\begin{proof}
It is well known that for a finite group $G$, the graph $\a$ is a spanning subgraph of the power graph $\mathcal{P}(G)$ and note that the lambda number is a monotone parameter. By {\rm \cite[Theorem 3.1]{a.Malambda}}, the result holds.
\end{proof}
\begin{theorem}{\label{Lambdanbrfinite}}
Let $G$ be a finite non-cyclic group of order $n$. Suppose $M_1, M_2, \ldots ,M_r$ be the maximal cyclic subgroups of $G$ such that $m_1\geq m_2 \geq \cdots \geq m_r $, where $m_i=\phi(|M_i|)$ for $1\leq i \leq r$. Then $$\lambda(\a) \leq  \left\{ \begin{array}{ll}
2n-|\g|-1; & \mbox{if $m_1\leq \sum \limits_{i=2}^r m_i$,}\\
2(n-m_1-1);& \mbox{ \text{Otherwise}. }\end{array} \right.$$
\end{theorem}
\begin{proof}
We prove this result by finding an upper bound of the path covering number of $\c$. We discuss the following two possible cases.

\noindent\textbf{Case-1:} $m_1\leq \sum \limits_{i=2}^r m_i$. We discuss this case into two subcases.

\textbf{Subcase-1.1:} $m_1=m_2$. Now we provide a Hamiltonian path in the subgraph of $\c$ induced by the set $\g$. Since the generators of two distinct maximal cyclic subgroups are adjacent in $\c$, note that the path $P  : x_{1,1}\sim x_{1,2}\sim \cdots \sim x_{1,r}\sim x_{2,1} \sim x_{2,2} \sim \cdots \sim x_{m_s, s}$, where $\langle x_{i,j} \rangle = M_j$, $1\leq i \leq m_j$ and $s= \text{max}\{t : 2\leq t \leq r ,  \ m_t=m_1\}$, covers all the vertices of $\g$ in $\c$. It follows that $c(\c) \leq n-|\g|+1$. By Theorem \ref{Lambdain path cover}, we have  $\lambda(\a)\leq 2n-|\g|-1$.


\textbf{Subcase-1.2:} $m_1 >m_2$. Since $\g= \bigcup \limits_{i=1}^r \mathcal{G}_{M_i}$, we consider $A_1=\{a_1,a_2,\ldots , a_{m_1-m_2}\} \subseteq \mathcal{G}_{M_1}$. In $A_2$, we collect the $m_1-m_2$ elements starting from $\mathcal{G}_{M_r}$. If $m_r\geq m_1-m_2$, then we take $A_2\subseteq \mathcal{G}_{M_r}$ such that $|A_2|=m_1-m_2$. Otherwise, we collect remaining $(m_1-m_2)-m_r$ elements from $\mathcal{G}_{M_{r-1}}$ and then choose remaining elements, if required, such that $|A_2|=m_1-m_2$, from $\mathcal{G}_{M_{r-2}},\mathcal{G}_{M_{r-3}}$ and so on. We write $A_2=\{b_1,b_2,\ldots ,b_{m_1-m_2}\}$. Further, consider the set $A_3=\g \setminus (A_1\cup A_2)$. In view of the above given partition of $\g$, now we provide a Hamiltonian path in subgraph of $\c$ induced by the set $\g$. 
Since the generators of distinct maximal cyclic subgroups are adjacent in $\c$. Thus, we have a Hamiltonian path $P: a_1\sim b_1 \sim a_2 \sim \cdots \sim a_{m_1-m_2} \sim b_{m_1-m_2}$ in the subgraph induced by the set $A_1\cup A_2$. Notice that the subgraph $\Gamma$ induced by the set $A_3$ in $\c$ is a complete $t$-partite graph, where $t=\mathrm{max}\{i: \mathcal{G}_{M_i}\cap A_3 \neq \varnothing \}$ and the partition set of $\Gamma$ is $\mathcal{G}_{M_1}\setminus {A_1}, \mathcal{G}_{M_2},\ldots , \mathcal{G}_{M_t}\setminus {A_2}$. Since $|\mathcal{G}_{M_1}\setminus {A_1}|=|\mathcal{G}_{M_2}|\geq |\mathcal{G}_{M_i}|$  for $3\leq i \leq t$, we have a Hamiltonian path $H'$ of $\Gamma$ with initial vertex $x$ belongs to $\mathcal{G}_{M_1}$. Since $b_{m_1-m_2}\in \mathcal{G}_{M_k}$ for some $k$, where $t\leq k \leq r$, we have $b_{m_1-m_2}\sim x$. Consequently, we get a Hamiltonian path in the subgraph induced by the set $\g$ in $\c$. Thus $c(\c) \leq n-|\g|+1$. By Theorem \ref{Lambdain path cover}, $\lambda(\a)\leq 2n-|\g|-1$. 

\noindent\textbf{Case-2:} $m_1> \sum \limits_{i=2}^r m_i$. Since $G$ is a non-cyclic group, it implies that $M_1$ is a proper subgroup of $G$. By consequence of Lagrange's theorem, $|M_1|\leq \frac{n}{2}$ and so $|\mathcal{G}_{M_1}|< \frac{n}{2}$. Notice that each element of $\mathcal{G}_{M_{1}}$ is adjacent to every element of $G\setminus M_1$ in $\c$. Thus, for $\langle x_i\rangle = M_1$ and $y_i\in G\setminus M_1$, we have a path $P:y_1\sim x_1\sim y_2 \sim  \cdots \sim x_{m_1}\sim y_{m_1+1}$ of length $2m_1+1$ in $\c$. Consequently,   $c(\c) \leq n-2m_1$. Hence, by Theorem \ref{Lambdain path cover}, $\lambda(\a)\leq 2n-2m_1-2$. 
\end{proof}
 In view of Lemma \ref{2maximal}, we have the following corollary of Theorem \ref{Lambdanbrfinite}.
 \begin{corollary}
 Let $G$ be a finite non-cyclic group of order $n$. Then $\lambda (\a) \leq 2n-4$, with equality holds if and only if $G$ is isomorphic to $\mathbb{Z}_2 \times \mathbb{Z}_2$.
 \end{corollary}
 \begin{proof}
 Since $G$ is a non-cyclic group, by Lemma \ref{2maximal}, we get $|\g|\geq 3$. Consequently, by Theorem \ref{Lambdanbrfinite}, $\lambda (\a) \leq 2n-4$. If $G$ is isomorphic to $\mathbb{Z}_2 \times \mathbb{Z}_2$, then $\lambda (\a)\geq 4=2o(G)-4$ and so $\lambda (\a)= 2n-4$. We now suppose that $\lambda (\a)=2n-4$. This is possible only when $|\g|=3$. Note that for a non-cyclic group $G$, $|\g|=3$ if and only if $G$ has exactly three maximal cyclic subgroups each with having only one generator. By Remark \ref{remark maximal}, $o(G)=4$. Thus, we must have $G\cong \mathbb{Z}_2\times \mathbb{Z}_2$.
 \end{proof}


 Now we classify finite simple groups $G$ such that $\lambda (\a)=|G|$. For this purpose, first we derive the following two lemmas. 
 
 \begin{lemma}{\label{pi G}}
 Let $G$ be a finite non-cyclic simple group. Then for any $d\in \pi_G$, there exists a Hamiltonian path in subgraph of $\c$ induced by the set  $\tau_d$.
 \end{lemma}
 \begin{proof}
  In view of Lemma \ref{cd2}, suppose $C_d=s$, where $ s\geq 2$. Let $\mathcal{H}_d=\{\rho _1, \rho _2,\ldots , \rho _s\}$ be the set of all cyclic classes of elements of order $d$. Let $x,y\in \rho _i$, where $i\in \{1,2,\ldots , s\}$. Then $x,y \in \langle x\rangle $ and so $x\sim y$ in $\a$. Consequently, $x\nsim y$ in $\c$. Also for $i\neq j$, let $x\in \rho _i$ and $y\in \rho _j$. Let $x\sim y $ in $\a$. Then there exists $z\in G$ such that $x,y \in\langle z\rangle$. Since $o(x)=o(y)=d$, we obtain $\langle x \rangle = \langle y \rangle$, which is not possible. Thus, $x\nsim y$ in $\a$ and so $x\sim y$ in $\c$. It follows that the subgraph of $\c$ induced by the set $\tau_d$ is a complete $s$-partite graph such that the size of each partition set is $\phi (d)$. Hence, the result holds.
 \end{proof}
 \begin{lemma}{\label{x join y}}
 Let $G$ be a finite non-cyclic simple group and $d_1,d_2 \in \pi_G$. For each $x \in \tau _{d_1}$, there exists $y\in \tau_{d_2}$ such that $x\sim y $ in $\c$.
 \end{lemma}
 \begin{proof}
 To prove this result, it is sufficient to show that there exists $y \in \tau _{d_2}$ such that $xy\neq yx$ so that $x\sim y$ in $\c$. Let  $xy'=y'x$ for all $y'\in \tau _{d_2}$. Note that $\langle \tau _{d_2} \rangle$ is a subgroup of $G$. For $g\in G$ and $x'=x_1x_2\cdots x_k \in \langle \tau _{d_2} \rangle$, where $x_i\in \tau _{d_2}$, we have $g^{-1}x'g= g^{-1}x_1x_2\cdots x_k g=g^{-1}x_1gg^{-1}x_2g\cdots gg^{-1}x_kg$. Since $g^{-1}x_ig\in \tau _{d_2}$ so that $g^{-1}x'g\in \langle \tau _{d_2} \rangle$. Therefore, $\langle \tau _{d_2} \rangle$ is  normal subgroup of $G$. But $G$ is simple and $\langle \tau _{d_2} \rangle \neq \{e\}$ gives $G=\langle \tau _{d_2} \rangle$. Since $x$ commutes with every element of $\tau _{d_2}$. It follows that $x$ belongs to the centre $Z(G)$ of $G$. Consequently, we get $G=Z(G)$ and so $G$ is a simple group which is abelian. Therefore, $G$ must be cyclic group of prime order; a contradiction. Thus, the result holds.
 \end{proof}
 \begin{theorem}
 Let $G$ be a non-trivial finite simple group of order $n$. Then $\lambda (\a)=n$ if and only if $G$ is not a cyclic group of order $n\geq 3$.
 \end{theorem}
 \begin{proof}
 If $G$ is cyclic, then by Theorem \ref{complete}, $\a$ is a complete graph. Consequently, $\lambda (\a)=2n-2$. Thus $\lambda (\a)=n$ if and only if $n=2$. We may now suppose that $G$ is a non-cyclic group. To prove our result, it is sufficient to show that the graph $\overline{\a \setminus \{e\}}$ has a Hamiltonian cycle (see Theorem \ref{Lambdain path cover}). Let $\pi _G = \{d_1, d_2, \ldots , d_k\}$. Then $G\setminus \{e\}=\bigcup \limits_{i=1}^{k} \tau _{d_i}$. By Lemma \ref{pi G}, for each $i\in [k]=\{1,2,\ldots, k\}$, we have a Hamiltonian path in the subgraph induced by the set $\tau _{d_i}$ in $\c$ and by Lemma \ref{x join y}, we get a Hamiltonian path in $\overline{\a \setminus \{e\}}$. Thus, the result holds.
 \end{proof}
 
 In the remaining part of the paper, we obtain the lambda number of enhanced power graphs of finite nilpotent groups. 
 
 \begin{lemma}{\label{cd3}}
  Let $G'$ be a non-trivial nilpotent group of odd order having no Sylow subgroup which is cyclic. If $G\cong G'\times \mathbb{Z}_n$, where $n\geq 1$ and $\mathrm{gcd}(n,|G'|)=1$, then  $C_{o(x)}\geq 3$ for each $x\in G\setminus \mathrm{Dom}(\a)$.
 \end{lemma}
 \begin{proof}
 Let $x=(x', y')$ be an arbitrary element of $G\setminus \mathrm{Dom}(\a)$. Since $G'$ is a nilpotent group, and so $G'=P_1\times P_2 \times \cdots \times P_r$, where $P_i^{\prime}s$ are Sylow subgroups of $G'$. Consequently, $x=(x_1, x_2,\ldots , x_r, y')$, where $x_i\in P_i$ for each $i\in [r]$. It follows that $x_j \neq e$ for some $j\in [r]$ because $x\notin \mathrm{Dom}(\a)$. Consider $y=(x_1, x_2,\ldots ,x_{j-1},y_j,x_{j+1}, \ldots , x_r, y')$, $z=(x_1, x_2,\ldots ,x_{j-1},z_j,x_{j+1}, \ldots , x_r, y')$, where $y_j, z_j\in P_j$ such that $o(x_i)=o(y_j)=o(z_j)$ and the cyclic subgroups $\langle x_j \rangle , \langle y_j \rangle , \langle z_j \rangle$ of $P_j$ are distinct [cf. Theorem \ref{pgroupclass}]. Clearly, $o(x)=o(y)=o(z)$. Note that the cyclic subgroups $\langle x \rangle , \langle y \rangle , \langle z \rangle$ of $G$ are distinct. Without loss of generality, let if possible, $\langle x \rangle = \langle y \rangle$. Then there exists $m\in \mathbb{N}$ such that $x^m=y$. Now consider $l=o(x_1)o(x_2)\cdots o(x_{j-1})o(x_{j+1})\cdots o(x_r)o(y')$. Then $x^{ml}=y^l$ and it follows that $x_j^{ml}=y_j^l$. Since $\mathrm{gcd}(o(y_j),l)=1$, we obtain $o(y_j)=o(y_j^l)$. Consequently, $\langle y_j\rangle =\langle y_j^l \rangle = \langle x_j ^{ml} \rangle \subseteq \langle x_j \rangle$. Thus, $\langle y_j\rangle = \langle x_j\rangle$; a contradiction. Thus, the result holds.
 \end{proof}
 
 \begin{lemma}{\label{hamiltonian path in tau d}}
 Let $G'$ be a non-trivial nilpotent group of odd order having no Sylow subgroup which is cyclic. If $G\cong G'\times \mathbb{Z}_n$, where $n\geq 1$ and $\mathrm{gcd}(n,|G'|)=1$, then for each $d\in D$, there exists a Hamiltonian path in the subgraph of $\c$ induced by the set $\tau _d$, where $D=\{o(x) : x\in G\setminus \mathrm{Dom}(\a)\}$.
 \end{lemma}
 
 \begin{proof}
 Let $d\in D$. Then by Lemma \ref{cd3}, $C_d=s$, where $s\geq 3$. Notice that the subgraph induced by the set $\tau _d$ in $\c$ is a complete $s$-partite graph with exactly $\phi(d)$ vertices in each partition set. Thus, we get a Hamiltonian path between any two elements of $\tau _d$.
 \end{proof}

 \begin{theorem}{\label{lambda number theorem odd}}
  Let $G'$ be a non-trivial nilpotent group of odd order having no Sylow subgroup which is cyclic. If $G\cong G'\times \mathbb{Z}_n$, where $n\geq 1$ and $\mathrm{gcd}(n,|G'|)=1$, then $\lambda(\a)= |G|+ |\mathrm{Dom}(\a)|-1$.
 \end{theorem}
 
 \begin{proof}
 In view of Theorem \ref{Lambdain path cover}, to prove our result it is sufficient to show that $c(\c)=|\mathrm{Dom}(\a)|+1$. For $x\in \mathrm{Dom}(\a)$, clearly $x$ is an isolated vertex in $\c$. Thus, it is sufficient to show that the subgraph of $\c$ induced by the non-dominating vertices of $\a$ has a Hamiltonian path. Let $G\setminus  \mathrm{Dom}(\a)$ has elements of order $d_1, d_2, \ldots , d_t$. Consider $S=\{d_1,d_2,\ldots ,d_t\}$ , where $d_1 < d_2< \cdots <d_t$.\\
{\rm  \textbf{Claim:}} There exists an ordered set $S'=\{\beta _1, \beta _2,\ldots , \beta _t\}$, where $\beta _i \in S$, such that either $\beta _i \vert \beta _{i+1}$ or $\beta _{i+1} \vert \beta _{i}$ for each $i\in [t-1]$.\\
\textit{Proof of claim:} If for each $i\in [t-1]$, either $d _i \vert d _{i+1}$ or $d _{i+1} \vert d _{i}$ then $S=S'$. Otherwise, choose the smallest $l$ such that neither $d_l\vert  d_{l+1}$ nor $d_{l+1}\vert  d_l$. By Lemma \ref{lcm}, $d_{l+j}=\mathrm{lcm}(d_l, d_{l+1})\in \pi_G$ for some $j\geq 2$. Now let $x=(x_1,x_2)\in G\setminus \mathrm{Dom}(\a)$ such that $o(x)=d_l$. Clearly, $o(x_1)>1$ and  $o(x)=o(x_1)o(x_2)$. Suppose $z=(z_1,z_2)\in G$ such that $o(z)=d_{l+j}$. Since $d_l\vert d_{l+j}$, it follows that $o(x_1)\vert o(z_1)$. Consequently, $o(z_1)>1$ and so $d_{l+j}\in S$ (cf. Theorem \ref{Dominating nilpotent}). Thus, by taking $i$-th element of the ordered set $\{d_1, d_2, \ldots , d_l, d_{l+j}, d_{l+1}, \ldots ,  d_{l+j-1}, d_{l+j+1}, \ldots ,d_t\}$ as $\gamma _i$, we get an ordered set $S_1=\{\gamma _1 ,\gamma _2, \ldots , \gamma _l, \gamma _{l+1}, \gamma _{l+2}, \ldots, \gamma _t\}$ such that either $\gamma _i \vert \gamma _{i+1}$ or $\gamma _{i+1} \vert \gamma _{i}$ for each $i\in \{1,2,\ldots, l+1\}$. If for each $i\in \{l+2,l+3,\ldots, t-1\}$, either $\gamma _i \vert \gamma _{i+1}$ or $\gamma _{i+1} \vert \gamma _{i}$, then $S_1=S'$. Otherwise, choose the smallest $l'\in \{l+2,l+3,\ldots, t-1\}$ and repeat the above process. On continuing this process, we get desired ordered set $S'$.

  
Now by Lemma \ref{hamiltonian path in tau d}, for each $i\in [t]$, the  subgraph of $\c$ induced by the set $\tau _{\beta _i}$ is a complete $s$-partite graph with $\phi(\beta _i)$ vertices in  each partition set. Observe that $G \setminus \mathrm{Dom}(\a) = \bigcup \limits_{i=1}^t \tau _{\beta _i}$. Then there exist paths $H_1, H_2,\ldots, H_t $ which covers all the vertices of $\tau_{\beta _1},\tau_{\beta _2} ,\ldots, \tau_{\beta _t}$, respectively. Now we shall show that for each $i\in [t-1]$, the end vertex of $H_i$ is adjacent to the initial vertex of $H_{i+1}$ in $\c$ through the following two cases:
  
   \noindent\textbf{Case-1: } $\beta _i\vert \beta _{i+1}$. Let $x\in \tau_{\beta _i}$ and $y$ be the initial vertex of $H_{i+1}$. If $x\nsim y$ in $\a$, then we choose $x$ to be the end vertex of $H_i$ so that $x\sim y \in \c$. Now we may assume that $x\sim y$ in $\a$. Then there exists $z\in G$ such that $x,y \in \langle z \rangle$. Since $\beta _i\vert \beta _{i+1}$, we get $\langle x\rangle \subset \langle y\rangle$. Let $x'\in \tau _{\beta _i}$ such that $\langle x\rangle \neq \langle x'\rangle$. Now if $x'\sim y$ in $\a$ then $\langle x'\rangle \subset \langle y\rangle$, which is not possible as $\langle x\rangle \neq \langle x'\rangle$. Thus $x'\sim y$ in $\c$. Therefore, we can choose $x'$ as the end vertex of $H_i$. 
   
    \noindent\textbf{Case-2: } $\beta _{i+1}\vert \beta _{i}$. Let $x$ be the end vertex of $H_i$ and $y\in \tau_{\beta _{i+1}}$. If $x\nsim y$ in $\a$, then we choose $y$ to be the initial vertex of $H_{i+1}$ so that $x\sim y \in \c$. Otherwise, there exists $z\in G$ such that $x,y \in \langle z \rangle$. Since $\beta _{i+1}\vert \beta _i$, it follows that $\langle y\rangle \subset \langle x\rangle$. Let $y'\in \tau _{\beta _{i+1}}$ such that $\langle y\rangle \neq \langle y'\rangle$. Now if $x\sim y'$ in $\a$ then $\langle y'\rangle \subset \langle x\rangle$, which is not possible as $\langle y\rangle \neq \langle y'\rangle$. Thus $x\sim y'$ in $\c$. Therefore, consider $y'$ as the initial vertex of $H_{i+1}$. 
   
  Hence, we get a Hamiltonian path in subgraph of $\c$ induced by the set $G \setminus \mathrm{Dom}(\a)$.
 \end{proof}
 
 \begin{lemma}{\label{cd3 even}}
  Let $G'$ be a nilpotent group of odd order having no Sylow subgroup which is cyclic. If $G\cong G'\times P \times \mathbb{Z}_n$, where $P$ is a non-cyclic $2$-group and $\mathrm{gcd}(n,|G'|)= \mathrm{gcd}(2,n)=1$,  then for each $x \in \mathcal{S}' = \{(g_1,g_2,g_3)\in G \; \mid \;  g_1\neq e_{G'}\}$, we have $C_{o(x)}\geq 3$.
 \end{lemma}
 
 \begin{proof}
After taking $\mathcal{S}'$ in place of $G\setminus \mathrm{Dom}(\a)$, the proof is similar to the proof of Lemma \ref{cd3}. Hence, we omit the details.
 \end{proof}
 
 \begin{theorem}{\label{Hamiltonian path in S'}}
    Let $G'$ be a nilpotent group of odd order having no Sylow subgroup which is cyclic. If $G\cong G'\times P \times \mathbb{Z}_n$, where $P$ is a non-cyclic $2$-group and $\mathrm{gcd}(n,|G'|)= \mathrm{gcd}(2,n)=1$, then there exists a Hamiltonian path in the subgraph of $\c$ induced by the set $\mathcal{S}'$.
 \end{theorem}
 
 \begin{proof}
  Using Lemma \ref{cd3 even}, and by the similar argument used in the proof of Theorem \ref{lambda number theorem odd}, the result holds.
 \end{proof}
 
 \begin{theorem}{\label{lambda number main}}
   Let $G'$ be a nilpotent group of odd order having no Sylow subgroup which is cyclic. If $G\cong G'\times P \times \mathbb{Z}_n$, where $P$ is a non-cyclic $2$-group and $\mathrm{gcd}(n,|G'|)= \mathrm{gcd}(2,n)=1$,  then $\lambda(\a)= |G|+ |\mathrm{Dom}(\a)|-1$.
 \end{theorem}
 \begin{proof}  In view of Theorem \ref{Lambdain path cover}, we show that the subgraph of $\c$ induced by the set of all non-dominating vertices of $\a$ has a Hamiltonian path. Let $e_1, e_2$ and $e_3$ be the identity elements of the groups $G', P$ and $\mathbb{Z}_n$, respectively. By Theorem \ref{Hamiltonian path in S'}, let $H'$ be a Hamiltonian path in the subgraph of $\c$ induced by the set $\mathcal{S}'$ and let  $g'=(x',y',z')$ be the end vertex of $H'$. By the proof of Theorem \ref{Hamiltonian path in S'}, notice that the order of $g'=(x',y',z')$ is maximum. Suppose that the exponent of the group $P$ is $2^k$. Consequently, $o(y')=2^k$. Further, observe that if $y\nsim y''$ in the graph $\mathcal{P}_E(P)$ then $(x,y,z)\nsim (x'',y'',z'')$ in $\a$. In view of Theorem \ref{pgroupclass} and Corollary \ref{pgroupclasscorollary}, we have the following cases:\\
 \noindent\textbf{Case-1: } \textit{$P$ is not of maximal class.} Consider the set $\mathcal{S}''=\{(e_1,y,z)\in G: y\neq e_2\}$. Notice that the sets $\mathcal{S}'$, defined in Lemma \ref{cd3 even}, $\mathcal{S}''$ and  $\mathrm{Dom}(\a)$ forms a partition of $G$. Now we provide a Hamiltonian path of the subgraph of $\c$ induced by the set $\mathcal{S}'\cup \mathcal{S}''$. To do this, first we drive a Hamiltonian path of the subgraph of $\c$ induced by the set $\mathcal{S}''$. By Theorem \ref{pgroupclass}, for $2\leq j \leq k$, we have $C_{2^j}\geq 2$ and $C_2\geq 3$ in $P$. For $1\leq j \leq k$, let $t_j=C_{2^j}$ and let $\mathcal{T}_j=\{\mathcal{C}_1^{(j)}, \mathcal{C}_2^{(j)}, \ldots , \mathcal{C}_{t_j}^{(j)}\}$ denotes the cyclic classes of $P$ containing the elements of order $2^j$.  Observe that each class in $\mathcal{T}_j$ is of cardinality $2^{j-1}$. Notice that each element of $P$  belongs to exactly one cyclic class of $\mathcal{T}_j$ for $j\in [k]$. Further note that, for $i\neq s$, if $x_i\in \mathcal{C}_i^{(j)}$ and $y_s\in  \mathcal{C}_s^{(j)}$, then $x_i\nsim y_s$ in $\mathcal{P}_E(P)$. We label a class in $\mathcal{T}_k$ by $\mathcal{C}_{1}^{(k)}$ such that $y'\notin \mathcal{C}_{1}^{(k)}$. For $2\leq j \leq k$, let $u_j\in \mathcal{C}_{t_j}^{(j)}$  be an arbitrary element. Now $u_j$ can be adjacent to at most one of the cyclic classes in $\mathcal{T}_{j-1}$ in $\mathcal{P}_E(P)$. If possible, let $u_j\sim v_1$ and $u_j\sim v_2$, where $v_1\in \mathcal{C}_{i_1}^{(j-1)}$, $v_2\in \mathcal{C}_{i_2}^{(j-1)}$. Then there exist elements $w_1, w_2\in P$ such that $u_j, v_1\in \langle w_1 \rangle$ and $u_j, v_2\in \langle w_2 \rangle$. Since $o(v_1)\vert o(u_j)$ and $o(v_2)\vert o(u_j)$, we get $\langle v_1 \rangle = \langle v_2 \rangle$. It follows that $i_1=i_2$. By $t_{j-1}\geq 2$, we obtain that the elements of $\mathcal{C}_{t_j}^{(j)}$ is not adjacent to at least one of the cyclic class in $\mathcal{T}_{j-1}$ in $\mathcal{P}_E(P)$. By relabelling, if necessary, we may assume that each element of $\mathcal{C}_{t_j}^{(j)}$  is not adjacent to every element of $\mathcal{C}_{1}^{(j-1)}$ in $\mathcal{P}_E(P)$. Since $t_1\geq 3$, we can label a class in $\mathcal{T}_1 \setminus \mathcal{C}_1^{(1)}$ by $\mathcal{C}_{t_1}^{(1)}$ in which for $x\in \mathcal{C}_{t_1}^{(1)}$ and $y\in \mathcal{C}_{1}^{(k)}$, we have $\langle x\rangle \nsubseteq \langle y \rangle$. It implies that $x\nsim y$ in $\mathcal{P}_E(P)$. Let $z_1,z_2,\ldots ,z_n$ be the elements of $\mathbb{Z}_n$. Then for $y_{p,q}^{(r)}$, the $q$-th element of $\mathcal{C}_{p}^{(r)}$, the path $H''$ given below\\
$(e_1,y_{1,1}^{(k)},z_1)\sim (e_1,y_{2,1}^{(k)},z_1)\sim \cdots \sim  (e_1,y_{t_k,1}^{(k)},z_1)\sim (e_1,y_{1,2}^{(k)},z_1)\sim (e_1,y_{2,2}^{(k)},z_1)\sim \cdots \sim (e_1,y_{t_k,2^{k-1}}^{(k)},z_1)\sim (e_1,y_{1,1}^{(k-1)},z_1) \sim (e_1,y_{2,1}^{(k-1)},z_1) \sim  \cdots \sim (e_1,y_{t_1,1}^{(1)},z_1) \sim (e_1,y_{1,1}^{(k)},z_2)\sim (e_1,y_{2,1}^{(k)},z_2)\ \sim \cdots \sim (e_1,y_{t_1,1}^{(1)},z_n)$, where $1\leq r \leq k$, $1\leq p \leq t_r$ and $1\leq q \leq 2^{r-1}$, is a Hamiltonian path in the subgraph of $\c$ induced by the set $\mathcal{S}''$. Since $y'\notin \mathcal{C}_1^{(k)}$, we get $y' \nsim y_{1,1}^{(k)}$ in $\mathcal{P}_E(P)$. Consequently, $(x',y',z')\nsim (e_1, y_{1,1}^{(k)}, z_1)$ in $\a$ and so $(x',y',z')\sim (e_1, y_{1,1}^{(k)}, z_1)$ in $\c$. Thus, we get a Hamiltonian path in the subgraph of $\c$ induced by the set $\mathcal{S}'\cup \mathcal{S}'' $.

  \noindent\textbf{Case-2: } \textit{$P$ is of maximal class.} In view of Corollary \ref{pgroupclasscorollary}, we discuss this case  into three subcases.

\textbf{Subcase-2.1:}
  $P=\mathbb{Q}_{2^{k+1}}=\left\langle x, y: x^{2^{k}}=e_2, x^{2^{k-1}}=y^{2}, y^{-1} x y=x^{-1}\right\rangle$, where $k \geq 2$.  
   Consider the set  $\mathcal{S}''=\{(e_1,b,c)\in G: b\neq e_2, x^{2^{k-1}}\}$. Note that the sets $\mathcal{S}', \  \mathcal{S}'' $ and $ \mathrm{Dom}(\a) $ forms a partition of the group $G$. Observe that $\mathbb{Q}_{2^{k+1}}$ has one maximal cyclic subgroup of order $2^k$ and $2^{k-1}$ maximal cyclic subgroup  of order $4$ (see \cite{a.dalal2021enhanced}).  Let $M'=\langle x\rangle $ be the maximal cyclic subgroup of order $2^k$ and let for $1\leq i \leq 2^{k-1}$, $M_i=\{e_2,x^{2^{k-1}},x^{i}y, x^{2^{k-1}+i}y\} $ be the maximal cyclic subgroups of order $4$. For $1\leq j \leq 2^{k-1}$, note that $x^jy$ is a generator of a maximal cyclic subgroup of $\mathbb{Q}_{2^{k+1}}$. Consequently, $x^jy\nsim a$ in $\mathcal{P}_E(\mathbb{Q}_{2^{k+1}})$, where $a\in \mathbb{Q}_{2^{k+1}} \setminus \langle x^jy\rangle$. Since $o(y')=2^k$, for $k\geq 3$, we have $y'\in M'$. Thus, the Hamiltonian path $H''$ in the subgraph of $\c$ induced by the set $\mathcal{S}''$ can be given as $ (e_1,xy,z_1)\sim (e_1,x,z_1)\sim (e_1,x^2y,z_1)\sim (e_1,x^2,z_1)\sim \cdots \sim  (e_1,x^{2^{k-1}-1}y,z_1)\sim (e_1,x^{2^{k-1}-1},z_1)\sim (e_1,x^{2^{k-1}}y,z_1)\sim (e_1,x^{2^{k-1}+1},z_1)\sim \cdots  \sim (e_1,x^{2^{k}-1},z_1) \sim (e_1,x^{2^{k}-1}y,z_1) \sim (e_1,y,z_1) \sim (e_1,xy,z_2)\sim (e_1,x,z_2)\sim \cdots \sim (e_1,y,z_n)$, where $z_1,z_2,\ldots ,z_n \in \mathbb{Z}_n$. We have a Hamiltonian path $H'$ in the subgraph induced by the set $\mathcal{S}'$ with end vertex $(x',y',z')$ and also $H''$ is a Hamiltonian path induced by $\mathcal{S}''$ with initial vertex $(e_1,xy,z_1)$. Moreover, $(x',y',z')\sim (e_1,xy,z_1)$. Thus, we get a Hamiltonian path $H$ in the subgraph induced by $\mathcal{S}' \cup  \mathcal{S}'' $.  If $k=2$ and $y'\in M_1$ then again we have a Hamiltonian path by interchanging the vertices $(e_1,xy,z_1)$ and $(e_1, x^2y, z_1)$ of $H$.
   
\textbf{Subcase-2.2:} $P=\mathbb{D}_{2^{k+1}}=\left\langle x, y: x^{2^{k}}=e_2= y^{2}, y^{-1} x y=x^{-1}\right\rangle $, where $k \geq 1$.
  Consider the set  $\mathcal{S}''=\{(e_1,b,c)\in G: b\neq e_2\}$. Observe that the sets  $\mathcal{S}' ,\ \mathcal{S}'' $ and $\mathrm{Dom}(\a) $ forms a partition of the group $G$. Also notice that $M'=\langle x\rangle $ is the only maximal cyclic subgroup  of order $2^k$ in $\mathbb{D}_{2^{k+1}}$ and for $1\leq i \leq 2^{k}$, $M_i=\{e_2,x^{i}y\} $ are the maximal cyclic subgroups of order $2$ in $\mathbb{D}_{2^{k+1}}$. By {\rm \cite[Figure 1]{a.panda2021enhanced}}, $x^jy$, where $1\leq j \leq 2^{k}$, is not adjacent to any  non-identity element of $\mathbb{D}_{2^{k+1}}$ in $\mathcal{P}_E(\mathbb{D}_{2^{k+1}})$. Since $o(y')=2^k$, for $k\geq 2$, we have $y'\in M'$. Thus, the Hamiltonian path in the subgraph of $\c$ induced by the set $\mathcal{S}''$ is $H'': (e_1,xy,z_1)\sim (e_1,x,z_1)\sim (e_1,x^2y,z_1)\sim (e_1,x^2,z_1)\sim \cdots \sim  (e_1,y,z_1) \sim (e_1,xy,z_2)\sim \cdots \sim (e_1,y,z_n)$, where $z_1,z_2,\ldots ,z_n \in \mathbb{Z}_n$. We have a Hamiltonian path $H'$ in the subgraph induced by the set $\mathcal{S}'$ with end vertex $(x',y',z')$ and also $H''$ is a Hamiltonian path induced by $\mathcal{S}''$ with initial vertex $(e_1,xy,z_1)$. Furthermore, $(x',y',z')\sim (e_1,xy,z_1)$. Consequently, we get a Hamiltonian path $H$ in the subgraph induced by $\mathcal{S}' \cup  \mathcal{S}'' $.  If $k=1$ and $y'\in M_1$, then again we have a Hamiltonian path by interchanging the vertices $(e_1,xy,z_1)$ in $(e_1, x^2y, z_1)$ of $H$.

\textbf{Subcase-2.3:} $P=\mathbb{SD}_{2^{k+1}}=\left\langle x, y: x^{2^{k}}=e_2, y^{2}=e_2, y^{-1} x y=x^{-1+2^{k-1}}\right\rangle$, where $k \geq 3$.
     Consider the set  $\mathcal{S}''=\{(e_1,b,c)\in G: b\neq e_2 \}$. Notice that the sets $\mathcal{S}', \ \mathcal{S}'' $ and  $\mathrm{Dom}(\a) $ forms a partition of the group $G$. Also note that $\mathbb{SD}_{2^{k+1}}$ has one maximal cyclic subgroup of order $2^k$, $2^{k-1}$ cyclic subgroup  of order $2$ and $2^{k-2}$ maximal cyclic subgroup of order $4$.  Let $M'=\langle x\rangle $ be the maximal cyclic subgroup of order $2^k$ and for $1\leq i \leq 2^{k-2}$, $M_{i}=\{e_2,x^{2^{k-1}},x^{2i+1}y, x^{{2^{k-1}}+2i+1}y\} $ be the maximal cyclic subgroups of $\mathbb{SD}_{2^{k+1}}$ of order $4$ and for $1\leq j \leq 2^{k-1}$, $M''_{j}=\{e_2, x^{2j}y\} $ be the maximal cyclic subgroups of $\mathbb{SD}_{2^{k+1}}$ of order $2$. Now for $1\leq t \leq 2^k$, $x^ty$ is a generator of a maximal cyclic subgroup of $\mathbb{SD}_{2^{k+1}}$. It follows that $x^jy\nsim b$ in $\mathcal{P}_E(\mathbb{SD}_{2^{k+1}})$, where $b\in \mathbb{SD}_{2^{k+1}}\setminus \langle x^jy \rangle $(see \rm \cite[Figure 2]{a.panda2021enhanced}). Let $z_1,z_2,\ldots ,z_n$ be the elements of $\mathbb{Z}_n$. Thus, the Hamiltonian path $H''$ in the subgraph of $\c$ induced by the set $\mathcal{S}''$ is as follows \\
      $(e_1,xy,z_1)\sim (e_1,x,z_1)\sim (e_1,x^3y,z_1)\sim (e_1,x^3,z_1)\sim   \cdots \sim  (e_1,x^{2^{k}-1}y,z_1)\sim (e_1,x^{2^{k}-1},z_1)\sim  (e_1,y,z_1)\sim (e_1,x^2y,z_1)\sim (e_1,x^2,z_1)\sim (e_1,x^4y,z_1)\sim \cdots  \sim (e_1,x^{2^{k}-2},z_1) \sim (e_1,x^{2^{k}-2}y,z_1) \sim (e_1,xy,z_2) \sim (e_1,x,z_2)\sim (e_1,x^3,z_2)\sim \cdots \sim (e_1,x^{2^{k}-2}y,z_n).$
     
     Moreover, we have a Hamiltonian path $H'$ in the subgraph induced by the set $\mathcal{S}'$ with end vertex $(x',y',z')$. Since $o(y')=2^k$, we have $y'\in M'$. It follows that $y'\nsim xy$ and so  $(x',y',z')\sim (e_1,xy,z_1)$ in $\c$. Thus, the subgraph of $\c$ induced by the set $G\setminus \mathrm{Dom}(\a) $ has a Hamiltonian path.
 \end{proof}
 
 \begin{corollary}
 For the group $G=\mathbb{Q}_{2^{k+1}}$, we have $\lambda (\a)=2^{k+1}+1$.
 \end{corollary}
 
 \begin{corollary}
 For the group $G\in \{\mathbb{D}_{2^{k+1}}, \mathbb{SD}_{2^{k+1}}\}$, we have $\lambda (\a)=2^{k+1}$.
 \end{corollary}
\section{Acknowledgement}
The first author gratefully acknowledge for providing financial support to CSIR  (09/719(0110)/2019-EMR-I) government of India. The second author 
gratefully acknowledge for Post Doctoral Fellowship (NISER/OO/SMS/PDF/2021-22/007) provided by the Department of Atomic Energy, Government of India.

\vspace{1cm}
\noindent
{\bf Parveen\textsuperscript{\normalfont 1}, \bf Sandeep Dalal\textsuperscript{\normalfont 2} Jitender Kumar\textsuperscript{\normalfont 1}}
\bigskip

\noindent{\bf Addresses}:

\vspace{5pt}



\begin{thebibliography}{10}

\bibitem{a.Cameron2016}
G.~Aalipour, S.~Akbari, P.~J. Cameron, R.~Nikandish, and F.~Shaveisi.
\newblock On the structure of the power graph and the enhanced power graph of a
  group.
\newblock {\em Electron. J. Combin.}, 24(3):3.16, 18, 2017.

\bibitem{a.Bera2017}
S.~Bera and A.~K. Bhuniya.
\newblock On enhanced power graphs of finite groups.
\newblock {\em J. Algebra Appl.}, 17(8):1850146, 2018.

\bibitem{a.bera2021connectivitydomin}
S.~Bera and H.~K. Dey.
\newblock On the proper enhanced power graphs of finite nilpotent groups.
\newblock {\em Journal of Group Theory}, 2022.
\newblock \href{doi:10.1515/jgth-2022-0057}{doi:10.1515/jgth-2022-0057}

\bibitem{a.bera2021connectivity}
S.~Bera, H.~K. Dey, and S.~K. Mukherjee.
\newblock On the connectivity of enhanced power graphs of finite groups.
\newblock {\em Graphs Combin.}, 37(2):591--603, 2021.

\bibitem{a.pgroupberkovi}
J.~G. Berkovi\v{c}.
\newblock {$p$}-groups of finite order.
\newblock {\em Sibirsk. Mat. \v{Z}.}, 9:1284--1306, 1968.

\bibitem{a.dalal2021enhanced}
S.~Dalal and J.~Kumar.
\newblock On enhanced power graphs of certain groups.
\newblock {\em Discrete Math. Algorithms Appl.}, 13(1):2050099, 2021.

\bibitem{b.dummit1991abstract}
D.~S. Dummit and R.~M. Foote.
\newblock {\em Abstract algebra}.
\newblock Prentice Hall, Inc., Englewood Cliffs, NJ, 1991.

\bibitem{a.Georges1995}
J.~Georges and D.~Mauro.
\newblock Generalized vertex labelings with a condition at distance two.
\newblock {\em Congr. Numer.}, 109:141--159, 1995.

\bibitem{a.Georger2003}
J.~P. Georges and D.~W. Mauro.
\newblock On regular graphs optimally labeled with a condition at distance two.
\newblock {\em SIAM J. Discrete Math.}, 17(2):320--331, 2003.

\bibitem{a.Georges1994}
J.~P. Georges, D.~W. Mauro, and M.~A. Whittlesey.
\newblock Relating path coverings to vertex labellings with a condition at
  distance two.
\newblock {\em Discrete Math.}, 135(1-3):103--111, 1994.

\bibitem{a.Griggs1992}
J.~R. Griggs and R.~K. Yeh.
\newblock Labelling graphs with a condition at distance 2.
\newblock {\em SIAM J. Discrete Math.}, 5(4):586--595, 1992.

\bibitem{a.Hale1980}
W.~Hale.
\newblock Frequency assignment: Theory and applications.
\newblock {\em Proceedings of the IEEE}, 68(12):1497--1514, 1980.

\bibitem{a.hamzeh2017automorphism}
A.~Hamzeh and A.~R. Ashrafi.
\newblock Automorphism groups of supergraphs of the power graph of a finite
  group.
\newblock {\em European J. Combin.}, 60:82--88, 2017.

\bibitem{b.pgroupisaac2006}
I.~M. Isaacs.
\newblock {\em Character theory of finite groups}.
\newblock AMS Chelsea Publishing, Providence, RI, 2006.
\newblock Corrected reprint of the 1976 original [Academic Press, New York;
  MR0460423].

\bibitem{kelarev2003graph}
A.~Kelarev.
\newblock {\em Graph algebras and automata}, volume 257 of {\em Monographs and
  Textbooks in Pure and Applied Mathematics}.
\newblock Marcel Dekker, Inc., New York, 2003.

\bibitem{a.kelarev2015}
A.~Kelarev, C.~Ras, and S.~Zhou.
\newblock Distance labellings of {C}ayley graphs of semigroups.
\newblock {\em Semigroup Forum}, 91(3):611--624, 2015.

\bibitem{a.kelarev2009cayley}
A.~Kelarev, J.~Ryan, and J.~Yearwood.
\newblock Cayley graphs as classifiers for data mining: the influence of
  asymmetries.
\newblock {\em Discrete Math.}, 309(17):5360--5369, 2009.

\bibitem{kelarev2002ring}
A.~V. Kelarev.
\newblock {\em Ring constructions and applications}, volume~9 of {\em Series in
  Algebra}.
\newblock World Scientific Publishing Co., Inc., River Edge, NJ, 2002.

\bibitem{kelarev2004labelled}
A.~V. Kelarev.
\newblock Labelled {C}ayley graphs and minimal automata.
\newblock {\em Australas. J. Combin.}, 30:95--101, 2004.

\bibitem{a.Kim2017}
B.~M. Kim, Y.~Rho, and B.~C. Song.
\newblock Lambda number for the direct product of some family of graphs.
\newblock {\em J. Comb. Optim.}, 33(4):1257--1265, 2017.

\bibitem{a.pgroupkulakoff}
A.~Kulakoff.
\newblock \"{U}ber die {A}nzahl der eigentlichen {U}ntergruppen und der
  {E}lemente von gegebener {O}rdnung in {$p$}-{G}ruppen.
\newblock {\em Math. Ann.}, 104(1):778--793, 1931.

\bibitem{a.Malambda}
X.~Ma, M.~Feng, and K.~Wang.
\newblock Lambda number of the power graph of a finite group.
\newblock {\em J. Algebraic Combin.}, 53(3):743--754, 2021.

\bibitem{a.ma2017perfect}
X.~Ma, R.~Fu, X.~Lu, M.~Guo, and Z.~Zhao.
\newblock Perfect codes in power graphs of finite groups.
\newblock {\em Open Math.}, 15(1):1440--1449, 2017.

\bibitem{a.ma2022survey}
X.~Ma, A.~Kelarev, Y.~Lin, and K.~Wang.
\newblock A survey on enhanced power graphs of finite groups.
\newblock {\em Electron. J. Graph Theory Appl.}, 10(1):89--111, 2022.

\bibitem{a.ma2020metric}
X.~Ma and Y.~She.
\newblock The metric dimension of the enhanced power graph of a finite group.
\newblock {\em J. Algebra Appl.}, 19(1):2050020, 2020.

\bibitem{a.pgroupmiller}
G.~A. Miller.
\newblock An {E}xtension of {S}ylow's {T}heorem.
\newblock {\em Proc. London Math. Soc. (2)}, 2:142--143, 1905.

\bibitem{a.mishra2021lambda}
M.~Mishra and S.~Sarkar.
\newblock Lambda numbers of finite $ p $-groups.
\newblock {\em \rm{arXiv}:2106.03916}, 2021.

\bibitem{a.panda2021enhanced}
R.~P.~Panda, S.~Dalal, and J.~Kumar.
\newblock On the enhanced power graph of a finite group.
\newblock {\em Comm. Algebra}, 49(4):1697--1716, 2021.

\bibitem{a.parveen2022}
Parveen, J.~Kumar, S.~Singh, and X.~Ma.
\newblock Certain properties of the enhanced power graph associated with a
  finite group.
\newblock {\em \rm{arXiv}:2207.05075}, 2022.

\bibitem{a.Roberts1991}
F.~S. Roberts.
\newblock {$T$}-colorings of graphs: recent results and open problems.
\newblock {\em Discrete Math.}, 93(2-3):229--245, 1991.

\bibitem{a.sarkar2022lambda}
S.~Sarkar.
\newblock The lambda number of the power graph of finite simple groups.
\newblock {\em \rm{arXiv}:2202.09818}, 2022.

\bibitem{b.westgraph}
D.~B. West.
\newblock {\em Introduction to graph theory}.
\newblock Prentice Hall, Inc., Upper Saddle River, NJ, 1996.

\bibitem{a.Whittlesey1995}
M.~A. Whittlesey, J.~P. Georges, and D.~W. Mauro.
\newblock On the {$\lambda$}-number of {$Q_n$} and related graphs.
\newblock {\em SIAM J. Discrete Math.}, 8(4):499--506, 1995.

\bibitem{a.Yeh2006}
R.~K. Yeh.
\newblock A survey on labeling graphs with a condition at distance two.
\newblock {\em Discrete Math.}, 306(12):1217--1231, 2006.

\bibitem{a.zahirovic2020study}
S.~Zahirovi\'{c}, I.~Bo\v{s}njak, and R.~Madar\'{a}sz.
\newblock A study of enhanced power graphs of finite groups.
\newblock {\em J. Algebra Appl.}, 19(4):2050062, 2020.

\bibitem{a.Zhou2005}
S.~Zhou.
\newblock Labelling {C}ayley graphs on abelian groups.
\newblock {\em SIAM J. Discrete Math.}, 19(4):985--1003, 2005.

\end{thebibliography}
\end{document}